\documentclass[11 pt]{amsart}
\usepackage{graphicx}
\usepackage[hidelinks]{hyperref}
\usepackage{color}
\usepackage{amssymb,amsmath,amsfonts,latexsym}
\usepackage{epstopdf}
\usepackage{nicefrac}
\usepackage{mathrsfs}
\usepackage{amsthm}
\usepackage{hyperref}
\hypersetup{
     colorlinks   = true,
     citecolor    = blue
}

\usepackage{tikz}

\usepackage{tikz,fullpage}
\usetikzlibrary{patterns}
\usepackage{enumerate}
\usepackage{ulem}
\usepackage{float, lscape}
\numberwithin{equation}{section}
\usepackage[margin=1.27in]{geometry}
\usetikzlibrary{arrows,%
                petri,%
                topaths}%
\usepackage{tkz-berge}
\usepackage[position=top]{subfig}
\usetikzlibrary{%
  matrix,%
  calc,%
  arrows%
}
\newtheorem{thm}{THEOREM}[section]

\newtheorem{lemma}[thm]{LEMMA}

\newcommand{\G}{\Gamma}

\newtheorem{definition}{Definition}

\begin{document}
\title{Non-commutative Poisson Structures on quantum torus orbifolds}
\author{Safdar Quddus}

\date{\today}
 
\let\thefootnote\relax

\begin{abstract}
We study the Hochschild cohomology and the Gerstenhaber algebra structure on the algebraic non-commutative torus/quantum torus orbifolds resulting by the action of finite subgroups of $SL_2(\mathbb Z)$. We also examine the Poisson structures and compute the Poisson cohomology.
\end{abstract}
\maketitle
\section{Introduction}

Inspired by this relationship between Poisson geometry and deformation theory, Jonathan
 Block and Ezra Getzler \cite{BG} and Ping Xu \cite{Xu} independently introduced a notion of a noncommutative Poisson structure and the non-commutative Poisson cohomology was studied by Xu explicitly for the smooth non-commutative 2-torus $\mathcal A_\theta$. The paper of Halbout and Tang \cite{HT} is the seminal paper for studying non-commutative Poisson structures on orbifolds or smash products.  \cite{HT} studied the Gestenhaber bracket and Poisson structure(s) on the $C^\infty (M) \rtimes G$, they showed that the Gerstenhaber bracket on the orbifold is a generalization of the classical Schouten-Nijenhuis bracket on manifold. This Gerstenhaber bracket on the orbifold, called as the ``twisted Schouten-Nijenhuis bracket" was used by them to understand the Poisson structures on the orbifold by solving $[ \Pi, \Pi ] =0$ for $\Pi \in H^2(C^\infty (M) \rtimes G, C^\infty (M) \rtimes G)$. This is how the non-commutative Poisson structure was studied for $\mathcal A\rtimes G$ for $\mathcal A = C^\infty(M)$. We can ask the question, what about such structures when $\mathcal A$ is a non-commutative (/smooth)manifold.\\

The graded Lie bracket on Hochschild cohomology remains elusive in contrast to the cup product. The latter may be defined via any convenient projective resolution. But, the former is defined
on the bar resolution, which is useful theoretically but not computationally, and one typically
computes graded Lie brackets by translating to another more convenient resolution via explicit
chain maps. Such chain maps are not always easy to find. Studying the Grestenhaber bracket on the Hochschild cohomology also gives us an insight to the deformations of the algebra, Hochschild cohomology of the $\mathcal A \rtimes G$, along with the Gerstenhaber bracket, is reflective of the deformation theory of modules, orbifold cohomology and Poisson cohomology. \cite{PPTT} \cite{LW1} \cite{LW2} \cite{BG}.

 Researchers like Witherspoon, Shepler, Negron and Zhou in various collaborations have studied the ``twisted Schouten-Nijenhuis bracket" for several variant spaces and discrete group actions, including the polynomial ring  $Sym(V)$ and the quantum polynomial ring $S_q(V)$ \cite{SW1} \cite{WZ} \cite{NW}. They used appropriate Koszul resolution to study the Gerstenhaber bracket structure. In the paper \cite{WZ} the discrete group action on the quantum symmetric polynomial algebra is considered for the special case when the action by a discrete group $G$ on it is diagonal$( v_i \mapsto \chi_g v_i$ for some character $\chi : G \rightarrow \mathbb C^{\times}$). The Koszul resolution for the quantum symmetric polynomial is due to Wambst\cite{W1} which he obtained by generalizing the Connes' Koszul resolution for non-commutative smooth 2-torus\cite{C} to quantum symmetric algebra $S_q(V)$ associated to a vector space $V$. The quantum torus also known as the quantum torus, can be thought as a quantum symmetric algebra in 2-variables but unlike Witherspoon and Zhou we shall be considering non-diagonal discrete group action on it arising by the restriction of the action of $SL_2( \mathbb Z)$ on the smooth non-commutative torus $\mathcal A_\theta$.

In fact Gerstenhaber bracket is a generalization of the usual Schouten-brackets of multivector fields \cite{T} and we saw in \cite{HT} that the Schouten-bracket gets twisted over orbifolds of smooth manifolds. Here, in this example, we shall see that over the non-commutative orbifolds such a twist can in fact be similarly understood over each twisted component arising by the paracyclic or the spectral decomposition of the cohomology.\\
 
Non-commutative torus was studied extensively by Connes and Rieffel in the 1980's and arose as a quantum deformation of the algebra of smooth functions on the classical 2-torus. It can also be seen as the irrational rotational algebra \cite{R}. Several invariants of the same were studied, which conformed the smoothness of the ``associated non-commutative space". The existence of non-commutative Poisson structure \cite{Xu} on the non-commutative torus establishes it as a preluding example to study noncommutative smooth spaces. What we study here is a dense sub-algebra of it. Let $\theta \in \mathbb R$, consider the group algebra $ \mathbb C<U_1^{\pm 1}, U_2^{\pm 1}>$ of the free group on two generators $U_1$ and $U_2$, with coefficients in $\mathbb C$. We define the associated quantum torus \cite{B} as
\begin{center}

$\mathscr A_\theta := \mathbb C<U_1^{\pm 1}, U_2^{\pm 1}>/<U_2U_1-\lambda U_1U_2>$
\end{center} for a fixed parameter $\lambda= e^{2 \pi i \theta} \in \mathbb C$. Unless specified otherwise, we will assume in this article that
\begin{equation}
\lambda^n \neq 1 \text{ for all } n\in \mathbb Z
\end{equation}

The algebra $\mathscr A_\theta$ may be thought of as a ring of functions on a quantum torus. Geometrically, the algebras $\mathcal A_\theta$, the non-commutative smooth 2-torus, arise as deformations of the ring $C^\infty(\mathbb T)$ of smooth functions on the two-dimensional torus $\mathbb T = S^1 \times S^1$, and as such, these are fundamental examples of noncommutative differentiable manifolds in the sense of Connes \cite{C}. On the other hand, algebraically, $\mathscr A_\theta$ is just a certain norm completion of $\mathcal A_\theta$. 
 
Xu studied the Poisson structure on non-commutative smooth torus $\mathcal A_\theta$\cite{Xu}. The unique 2-cocycle of the non-commutative smooth torus $\mathcal A_\theta$ when $\theta$ satisfies the Diophantine condition gave the Poisson structure. The Poisson cohomology groups were also studied therein, the non-trivial Poisson 2 cocycles give rise to formal deformations of the algebra. This is how the formal deformations are studied and further one can ask about the index theory of such spaces.\\
  
We shall study the non-commutative Poisson structure on the orbifolds of $\mathscr A_\theta$ arising through the action of finite subgroups of $SL_2(\mathbb Z)$. The action of $SL_2(\mathbb Z)$ on $\mathscr A_\theta$ is described as follows. An element 
\begin{center}
$g= \left[
 \begin{array}{cc}
   g_{1,1} & g_{1,2} \\
   g_{2,1} & g_{2,2}
 \end{array} \right]\in SL_2(\mathbb Z)$
\end{center} acts on the generators $U_1$ and $U_2$ as described below
\begin{center}
$\rho_gU_1=e^{(\pi i g_{1,1} g_{2,1})\theta}U_1^{g_{1,1}}U_2^{g_{2,1}} \text{ and }\rho_gU_2=e^{(\pi i g_{1,2} g_{2,2})\theta}U_1^{g_{1,2}}U_2^{g_{2,2}}$\\.
\end{center}
We leave it to the readers to check that the above is indeed an action of $SL_2(\mathbb Z)$ on the algebra $\mathscr A_\theta$, in fact as we shall later see the action preserves the Calabi-Yau algebraic structure of $\mathscr A_\theta$. We illustrate this action by considering the action of $\mathbb Z_4$ on the algebraic noncommutative torus, $\mathscr A_\theta$. We notice that the generator of $\mathbb Z_4$ in $SL_2(\mathbb Z)$, $g= \left[
 \begin{array}{cc}
   0 & -1 \\
   1 & 0
 \end{array}
\right]$ acts on $\mathscr A_\theta$ as follows
\begin{center}
$ \left[
 \begin{array}{cc}
   0 & -1 \\
   1 & 0
 \end{array}
\right] U_1 = U_2^{-1}$ and  $\left[
 \begin{array}{cc}
   0 & -1 \\
   1 & 0
 \end{array}
\right] U_2 = U_1$.
\end{center}

\section{Gestenhaber bracket and the Poisson Structure}
For any algebra $A$ over a field $k$, Hochschild cohomology $HH^\bullet(A)$ is the space $Ext^\bullet_{A\otimes A^{op}}(A,A)$, which has two compatible operations, cup product and bracket. Both operations are defined initially on the bar resolution, a natural $A \otimes A^{op}$-free resolution of $A$. It is a classical fact that the Hochschild cohomology $HH^\ast(A,A) := \bigoplus_{\bullet \in \mathbb Z}H^\bullet(A, A)$ of an associative algebra $A$ carries a Gerstenhaber algebra structure. A Gerstenhaber algebra is a graded associative algebra $(H^\ast = \bigoplus_{i \in \mathbb Z} H^i, \cup)$ together with a degree $-1$ graded Lie bracket $[-, -]$ compatible with the product $\cup$ in the sense of the following Leibniz rule
\begin{equation}
[a \cup b, c] = [a,c] \cup b + (-1)^{(|c|-1)|a|} a\cup [b,c]
\end{equation}
For an associative algebra $(A, \mu)$, the Hochschild co-chain groups $C^\bullet(A,A)$ (cup) product
$\cup : C^m(A, A) \otimes C^n(A, A) \rightarrow C^{m+n}(A, A)$ defined by
\begin{equation}
(f \cup g)(a_1, . . . , a_{m+n}) = \mu(f(a_1, . . . , a_m), g(a_{m+1}, . . . , a_{m+n})).
\end{equation}

It turns out that the Hochschild coboundary operator $\delta$ is a graded derivation with respect to the cup product. Hence, it induces a cup product $\cup$ on the Hochschild cohomology $HH^\ast(A,A)$. Moreover, the co-chain groups $C^\bullet(A, A)$ carry a degree $-1$ graded Lie bracket compatible with the Hochschild co-boundary $\delta$ \cite{G}. Therefore, it gives rise to a degree $-1$ graded Lie bracket on $HH^\ast(A,A)$. The cup product and the degree $-1$ graded Lie bracket on the Hochschild cohomology $HH^\ast(A, A)$ are compatible in the sense of (1.1) to make it into a Gerstenhaber algebra.\\  

Using the Gerstenhaber bracket on $HH^\ast(C^\infty(M))$, one can define the Poisson structure on the manifold $M$, a Poisson structure on a smooth manifold $M$ is a Lie bracket $\{,\}: C^\infty(M) \times C^\infty(M) \rightarrow C^\infty(M)$ satisfying 
\begin{center}
${\displaystyle \{f,gh\}=\{f,g\}h+g\{f,h\}}$
for $f,g, h \in C^\infty(M)$.
\end{center}
For $\Pi \in H^2(C^\infty(M))$ and $[\Pi, \Pi]=0$, the Poisson bracket on $M$ can be defined as 
\begin{center}
$\{f,g\} := [[\Pi, f],g]$.
\end{center}

The above characterization can be generalized to an associative algebra. Hence, the non-commutative Poisson structure of an associative algebra $A$ can be defined as follows:

\begin{definition}
 Let $A$ be an associative algebra. A Poisson
structure on $A$ is an element $\Pi \in H^2(A,A)$ such that $[\Pi, \Pi ]=0$. An algebra with a
Poisson structure is called a Poisson algebra.

\end{definition}
\section{Poisson Structure on non-commutative Orbifolds}

In this section we discuss about the Gerstenhaber brackets and Poisson structure on a general non-commutative orbifold. For a non-commutative manifold $\mathcal A$ and a discrete group $G$ acting on it, we can consider the orbifold $\mathcal A \rtimes G$. To understand the Gerstenhaber bracket and the possible Poisson structure on $\mathcal A \rtimes G$ we have to understand its Hochschild cohomology, $HH^\bullet(\mathcal A \rtimes G, \mathcal A \rtimes G)$.\\

The Hochschild cohomology splits as \cite{GJ}

\begin{equation}
HH^\bullet(\mathcal A \rtimes G, \mathcal A \rtimes G) = \bigoplus_{\gamma \in G} HH^\bullet(\mathcal A, {}_{\gamma}\mathcal A)^G.
\end{equation}
Where ${}_{\gamma}\mathcal A$ is set wise $\mathcal A$ with the left $\mathcal A$-module structure defined as $\alpha \cdot a = \gamma (\alpha)a$ where $\alpha \in \mathcal A$ and $a \in {}_{\gamma}\mathcal A$ \cite{Q1}.
Witherspoon and Shepler \cite{SW2} studied the same and constructed explicit isomorphism between the two cohomology groups described above. The general, method is to decompose into the conjugacy classes of $G$ and using an appropriate projective resolution for $\mathcal A$. Using the resolution the Gerstenhaber bracket described over the bar resolution is transferred via chain map to the convenient resolution and an explicit formulation of the Gerstenhaber bracket is then obtained. The article \cite{SW2} had studied the Gestenhaber bracket over the polynomial skew group algebra and provided necessary conditions to find possible Poisson structure for the polynomial skew group algebra.\\

We shall understand the Gerstenhaber bracket and then find the Poisson structure(s) and finally the associated Poisson cohomology for the non-commutative orbifolds arising from the action of the four finite subgroups of $SL_2(Z)$ on the quantum torus.

\section{Connes' Koszul Resolution}
We briefly recall the Connes' Koszul resolution in this section. It is a projective resolution for the non-commutative torus and the quantum torus.  Wambst \cite{W1} generalized it to the higher dimensional quantum torus and later computed its homology \cite{W2}.
\begin{center}
$\mathscr A_\theta \xleftarrow{\epsilon} (\mathscr A_\theta)^e \otimes \mathbb C \xleftarrow{b_1} (\mathscr A_\theta)^e \otimes <e_1>\bigoplus (\mathscr A_\theta)^e\otimes <e_2> \xleftarrow{b_2} (\mathscr A_\theta)^e \otimes <e_1 \wedge e_2>$
\end{center}
where, $(\mathscr A_\theta)^e = \mathscr A_\theta \otimes (\mathscr A_\theta)^{op}$\\
$\epsilon(a\otimes b ) = ab;
b_1(1\otimes e_j)= 1\otimes {U_j}- {U_j}\otimes 1 \text{ and } \newline
b_2(1\otimes( e_1 \wedge e_2 ) ) = (U_2\otimes 1 - \lambda \otimes U_2 )\otimes e_1- ( \lambda U_1\otimes 1 - 1\otimes U_1 )\otimes e_2.$\\
We shall use the above resolution to study the Hochschild cohomology $HH^\bullet({}_{\gamma} \mathscr A_\theta)$, which is the cohomology of the following complex:

$${}_{\gamma}\mathscr A_\theta \xrightarrow{{}_{\gamma}\alpha_1}{}_{\gamma}\mathscr A_\theta \oplus {}_{\gamma}\mathscr A_\theta \xrightarrow{{}_{\gamma}\alpha_2} {}_{\gamma}\mathscr A_\theta \rightarrow 0,$$
where the maps are as below:
$${}_{\gamma}\alpha_1(\varphi) = ((\gamma \cdot U_1)\varphi - \varphi U_1,\displaystyle (\gamma\cdot U_2) \varphi - \varphi U_2)\; ;{}_{\gamma}\alpha_2(\varphi_1, \varphi_2) = \displaystyle (\gamma \cdot U_2)\varphi_1-\lambda \varphi_1 U_2 - \lambda (\gamma\cdot U_1) \varphi_2 + \varphi_2 U_1.$$ We leave it upon the readers to verify these maps, which can be done easily using the Connes' Koszul resolution.

\section{Some Homology groups of $\mathcal A_\theta \rtimes \Gamma$}

\begin{lemma}
$HH^0(\mathcal A_\theta \rtimes \Gamma)  = \mathbb C $\\
\end{lemma}
\begin{proof}
Using the resolution, we have , $HH^0(\mathcal A_\theta) = \mathbb C$. It is generated by the $a_{0,0} \otimes 1 \otimes 1$ and since the chain maps at the zeroth cohomology are identity maps, the cocycle is $\G$ invariant. Similarly, we see that $HH^0({}_{\gamma}\mathcal A_\theta)$ is the kernel of the map ${}_{\gamma}\alpha_1$. The result is a routine calculation and the invariance is also a straightforward computation.
\end{proof}
\begin{lemma}
For $\theta$ satisfying Diophantine condition, $HH^2(\mathcal A_\theta)^ \G \cong \mathbb C$.
\end{lemma}
\begin{proof}

We know from \cite[Lemma 4.1]{Xu} that $HH^2(\mathcal A_\theta) \cong \mathbb C$. To check its invariance under the action of $\G$, we push the cocycle into the bar complex using the map $h_2$ \cite{C} and after the action by the generator of $\G$ we pull it back on to the Connes' Koszul resolution and compare the equivalence classes. It can be easily checked that $HH^2(\mathcal A_\theta)^\G \cong \mathbb C$. We shall explicitly calculate for $\mathbb Z_3$, other cases too have similar computation. We abide by the notations of \cite{C}.\\

For $\varphi \in {}_{\omega}\mathcal A_\theta$, let $\widetilde{\varphi}$ be the corresponding element of $\text{Hom}_{\mathcal A_\theta^{e}}(\Omega_2,{}_{\omega}\mathcal A_\theta)$. Then
$$\widetilde{\varphi}(a\otimes b \otimes e_1 \wedge e_2)(x)=\varphi((\omega\cdot b)xa),$$
for all  $a,b,x \in \mathcal A_\theta$.
Let $\psi = k_2^{\ast} \widetilde{\varphi} = \widetilde{\varphi} \circ k_2$. We have
$$\psi(x,x_1,x_2)=\widetilde{\varphi}(k_2(I \otimes x_1 \otimes x_2))(x),$$
for all $x,x_1,x_2 \in \mathcal A_\theta$.
The group $\omega$ acts on $\mathcal A_\theta$ in the bar complex as $$\omega \cdot \chi(x,x_1,x_2)=\chi(\omega\cdot x,\omega \cdot x_1, \omega \cdot x_2).$$ 
Further we pull the map ${}_{\omega}\psi := \omega \cdot \psi$ back on to the Connes complex via the map $h_2^{\ast}$. Let $w=h_2^{\ast}({}_{\omega}\psi)$ denote the pull-back of ${}_{\omega}\psi$ on the Connes' Koszul complex. We have  
\begin{center}
$w(x)={}_{\omega}\psi(x,U_2,U_1)-\lambda{}_{\omega}\psi(x,U_1,U_2)=\psi(\omega \cdot x, \frac{U_1U_2^{-1}}{\sqrt\lambda},U_2^{-1}) -\lambda \psi(\omega \cdot x, U_2^{-1},\frac{U_1U_2^{-1}}{\sqrt\lambda})= \newline
\widetilde{\varphi}(k_2(I \otimes \frac{U_1U_2^{-1}}{\sqrt\lambda}\otimes U_2^{-1}))(\omega\cdot x)-\lambda \widetilde{\varphi}(k_2(I \otimes U_2^{-1}\otimes \frac{U_1U_2^{-1}}{\sqrt\lambda}))(\omega\cdot x)$.
\end{center}
Using the results of \cite{C} and \cite[Section 6]{Q1}, we have
$$k_2(I \otimes U_1U_2^{-1}\otimes U_2^{-1})-\lambda k_2(I \otimes U_2^{-1}\otimes U_1U_2^{-1})= (U_2^{-1} \otimes U_2^{-2}).$$
Hence,
\begin{center}
$ \displaystyle\frac{1}{\sqrt\lambda}(\widetilde{\varphi}((k_2(I \otimes U_1U_2^{-1}\otimes U_2^{-1}))(\omega\cdot x)-\lambda \widetilde{\varphi}(k_2(I \otimes U_2^{-1}\otimes U_1U_2^{-1})))(\omega\cdot x))=\displaystyle\frac{1}{\sqrt\lambda}\widetilde{\varphi}((U_2^{-1} \otimes U_2^{-2}))(\omega\cdot x) = {\sqrt\lambda}\varphi(U_1^{-2 }U_2^{2} \cdot (\omega \cdot x)\cdot U_2^{-1})$.
\end{center}
The cocycle $x_{-1,-1} \in H^2(\mathcal A_\theta, \mathcal A_\theta)$ which on the Connes' Koszul complex is supported at the ${(-1,-1)}$ is invariant under the action of $\mathbb Z_3$ because: \par

\begin{center}
$\displaystyle\frac{1}{\sqrt\lambda}x_{-1,-1}(U_2^{-2} (\omega\cdot x)  U_2^{-1})= \displaystyle\frac{1}{\sqrt\lambda}\varphi_{-1,-1}(U_2^{-2} (\omega\cdot (x_{-1,-1} U_1^{-1} U_2^{-1}))  U_2^{-1})=\displaystyle\frac{1}{\sqrt\lambda}\varphi_{-1,-1}(U_2^{-2} x_{-1,-1} U_2 {(\frac{U_1 U_2^{-1}}{\sqrt\lambda})}^{-1}  U_2^{-1})=\varphi_{-1,-1}(x_{-1,-1} U_2^{-2} U_2 U_2 U_1^{-1} U_2^{-1}) = \varphi_{-1,-1}(x_{-1,-1} U_1^{-1} U_2^{-1}) = x_{-1,-1}$.
\end{center}
Hence, $HH^2(\mathcal A_\theta)^{\mathbb Z_3} \cong \mathbb C$.\\

\end{proof}

\section{Hochschild cohomology $HH^\bullet(-)$ of $\mathscr A_\theta \rtimes \G$}
\begin{thm}
Let $\Gamma$ be any finite subgroup of $SL_2(\mathbb Z)$ then we have the following:\\
$HH^0(\mathscr A_\theta \rtimes \G) = \mathbb C$, $HH^1(\mathscr A_\theta \rtimes \G) = 0$, $HH^\bullet(\mathscr A_\theta \rtimes \Gamma) = 0 \text{ for } \bullet > 2$.\\ and \\
$HH^2(\mathscr A_\theta \rtimes \Gamma)  \cong\begin{cases}
\mathbb C^5  & \text{ for } \G = \mathbb Z_2 \\
\mathbb C^7 & \text{ for } \G = \mathbb Z_3\\
\mathbb C^8  & \text{ for } \G = \mathbb Z_4 \\
\mathbb C^9 & \text{ for } \G = \mathbb Z_6. \end{cases}$\\

\end{thm}

\begin{proof}[Proof of Theorem 5.1]

The Hochschild cohomology decomposes as follows: 
$$HH^\bullet(\mathscr A_\theta \rtimes \G) = \bigoplus_{\gamma \in \G} H^\bullet(\mathscr A_\theta, {}_{\gamma}\mathscr A_\theta)^\G.$$

It is clear from the Koszul resolution that the cohomology  $HH^\bullet((\mathscr A_\theta \rtimes \G))$ for $\bullet > 2$ vanishes. The second cohomology $HH^2(\mathscr A_\theta \rtimes \G)$ has the following decomposition:

$$HH^2(\mathscr A_\theta \rtimes  \G) \cong \bigoplus_{\gamma \in \G} H^2({}_{\gamma}\mathscr A_\theta)^\G.$$

We have $HH^2(\mathscr A_\theta)^\G = \mathbb C$, the proof is similar to the proof in the last section for the smooth orbifolds.\\

To calculate $HH^2({}_{\gamma}\mathscr A_\theta)^\G$ we firstly study $HH^2({}_{\gamma}\mathscr A_\theta) = {}_{\gamma}\mathscr A_\theta / im({}_{\gamma}\alpha_2)$ on the Connes' Koszul complex. It is a straight forward computation that shows that for $\gamma \in \G$, on the Koszul complex, the group $HH^2({}_{\gamma}\mathscr A_\theta)$ is generated by equivalence classes of elements of the form $$a_{p,q} \otimes 1 \otimes e_1 \wedge e_2 \in {}_{\gamma}\mathscr A_\theta \otimes (\mathscr A_\theta)^e \otimes \Omega_2.$$
where $p,q \in  \{0,1\}$. For example, for $g=\sqrt{-1} \in \mathbb Z_4$; $HH^2({}_{\gamma}\mathscr A_\theta)$ is two dimensional and is generated by $a_{0,0} \otimes 1 \otimes e_1 \wedge e_2$ and $a_{1,0} \otimes 1 \otimes e_1 \wedge e_2$. \\

Similar to the method we adopted, we can check that all the $\gamma$-twisted 2-cocycles are $\G$ invariant. In general for a commutative algebra $A$ and some $d \in \mathbb N$, $HH^d(A) \cong HH_{n-d}(A)$ if $A$ is Gorenstein \cite{V} and for non-commutative algebras that are Calabi-Yau such a duality is established \cite{CYZ}. We see that the Hochschild homology of $\mathscr A \rtimes \G$ \cite{Q1} and the cohomology computed above are dual to each other in dimension. This gives us an example of an algebra whose Van den Bergh duality is preserved upon the action of $\G$. In general one can ask if given a Calabi Yau algebra and a finite group acting on it, is the skew group algebra Calabi Yau? Recently this phenomenon was addressed by Meur \cite{M}.

The dimension of the zeroth cohomology can be concluded from the previous section. To determine the $HH^1(\mathscr A_\theta \rtimes \G)$ we firstly observe that $HH^1(\mathscr A_\theta)$ is two dimensional and is generated by the equivalence class of elements supported at $\phi^1_{-1,0}$ and $\phi^2_{0,-1}$. Using the maps $k_1$ and $h_1$ we check the invariance of these two cocycles in the similar way that we did for the 2-cocycle above. We can easily observe that $HH^1(\mathscr A_\theta)^\G=0$ for all $\G$. We now claim that $HH^1({}_{\gamma}\mathscr A_\theta) = 0$ for all $\gamma \in \G$. This statement is an straightforward corollary of the fact that $HH^1(\mathscr A_\theta, {}_{\gamma}\mathscr A_\theta^\ast) = 0$ for all $\gamma \in \G$ (see  \cite[Lemma 2.6]{Q3} ).
\end{proof}

\section{Gestenhaber Bracket for $\mathscr A_\theta \rtimes \G$}
We have described the above the method to compute the Gerstenhaber bracket by transporting the Gerstenhaber structure from the bar complex to convenient complex(here the $\gamma$-twisted Koszul complex) and then gain push forward to the bar complex for the explicit formula. It is to be noted that for $f \in HH^n(A,{}_{\gamma}A)$ and $g \in HH^m(A,{}_{\mu}A)$, for $\gamma$, $\mu \in \G$, the Gerstenhaber bracket $[f,g]$ is a $n+m-1$ cocycle. We can hence conclude that for any $\G < SL_2(\mathbb Z)$, finite subgroup, the Gestenhaber bracket on $HH^\bullet(\mathscr A_\theta \rtimes \G)$ is trivial.\\

\begin{lemma}
The Gerstenhaber bracket on the non-commutative torus orbifolds $\mathscr A_\theta \rtimes \G$ for $\G$ finite subgroups of $SL_2(\mathbb Z)$ is 0. That is:
$$ [-,-]: HH^\ast(\mathscr A_\theta  \rtimes \G) \rightarrow HH^\ast(\mathscr A_\theta  \rtimes \G).$$
is the zero map. 
\end{lemma}
\begin{proof}
Any non-trivial 2-cocycle, $\Pi^\G_i$, when paired with any of the $0$-cocycles should yield a 1-cocycle in $HH^1(\mathscr A_\theta \rtimes \G)(= 0)$. We denote by $\Pi^\Gamma_i$ a 2-cocycle in $\mathscr A_\theta \rtimes \G$. Enough to show that $[\Pi^\G_i, \Pi^\G_i]=0$, since $[\Pi^\G_i, \Pi^\G_i]$ is a 3-cocycle in $HH^\ast(\mathscr A_\theta \rtimes \G)$ hence 0.\\

We note that the untwisted 2-cocycle $\Pi_0 =  \delta_1 \wedge \delta_2$ where $\delta_1$ and $\delta_2$ are the two canonical derivations on $\mathscr A_\theta$ given by,
$$\delta_1(U_1^n U_2^m) = 2 \pi i n U_1^n U_2^m$$ and $$\delta_2(U_1^n U_2^m) = 2 \pi i m U_1^n U_2^m$$ 
 
\end{proof}

\section{Poisson Structure(s) and cohomology of $\mathscr A_\theta \rtimes \G$}
The Poisson cohomology for a Poisson algebra is defined as follows. For an associative Poisson algebra $(A, \Pi)$, we define a co-chain complex $(HH^\ast(A,A),d_{\Pi})$, 
$$d_{\Pi}:HH^\bullet(A,A) \rightarrow HH^{\bullet +1}(A,A)$$ where
$$d_{\Pi}(U)= [\Pi, U], \text{ for all }U \in  HH^i(A,A).$$
From the condition that $[\Pi, \Pi] = 0$ and the graded Jacobi identity of the Gerstenhaber brackets, it follows that $d_{\Pi}^2 = d_{\Pi} \circ d_{\Pi} = 0$. The cohomology of this complex $(HH^\ast(A,A), d_{\Pi})$ is called Poisson cohomology of $(A,\Pi)$ and denoted by $H_{\Pi}^\ast(A)$.\\

Since the Gerstenhaber bracket is 0, all the 2-cocycles are Poisson structures on $\mathscr A_\theta \rtimes \G$. 
\begin{thm} For all $i$,
$H_{\Pi^\G_i}^0(\mathscr A_\theta \rtimes \G) = \mathbb C$, $H_{\Pi^\G_i}^1(\mathscr A_\theta \rtimes \G) = 0$ and \\
$H_{\Pi^\G_i}^2(\mathscr A_\theta \rtimes \Gamma)  \cong\begin{cases}
\mathbb C^5  & \text{ for } \G = \mathbb Z_2 \\
\mathbb C^7 & \text{ for } \G = \mathbb Z_3\\
\mathbb C^8  & \text{ for } \G = \mathbb Z_4 \\
\mathbb C^9 & \text{ for } \G = \mathbb Z_6. \end{cases}$\\
$H_{\Pi^\G_i}^\bullet(\mathscr A_\theta \rtimes \Gamma) = 0 \text{ for } \bullet > 2$.\\
\end{thm}
\begin{proof}
The proof is straight forward consequence of Theorem 5.1 and Lemma 6.1.
\end{proof} 
\section{Conclusion}
Although we had the machinery to compute the Gestenhaber brackets explicitly, the brackets are trivial for all the quantum orbifolds.  The vanishing of the $HH^1(\mathscr A_\theta \rtimes \G)$ is the key to isomorphism of the Hochschild cohomology and the Poisson cohomology. The presence of $\gamma$-twisted 2 cocycles in the Poisson cohomology $HH^2_{\Pi}(\mathscr A_\theta \rtimes \G)$ indicates interesting deformation theory for orbifolds. It is well known that for a Poisson algebra $(A, \Pi)$ and a Poisson 2 cocycle $\Pi_1 \in HH^2_{\Pi}(A)$, the group $HH^3_{\Pi}(A)$ is the obstruction for existence of a deformation of $\Pi$ with infinitesimal $\Pi_1$. So to conclude, the formal deformations in the sense of Kontsevich\cite{K} which was studied by Xu\cite{Xu}, Neumaier et al.\cite{NPPT} amongst others does exist in these non-commutative orbifolds along with that which was inherited from the parent non-commutative manifold $\mathcal A_\theta$.\\

\section{Acknowledgement}
I acknowledge the discussion with Prof. Xiang Tang and his valuable comments.

\vspace{2mm}

{\small \noindent{Safdar Quddus},\\ Department of Mathematics,\\ Indian Institute of Science, Bengaluru, India.\\
Email: safdarquddus@iisc.ac.in.


\begin{thebibliography}{}

\bibitem[B]{B} Baudry, J.;Invariants du tore quantique. (French. French summary) [Invariants of the quantum torus]; Bull. Sci. Math. 134 (2010), no. 5, 531-547.

\bibitem[BG]{BG} Block, J. and Getzler, E.; Quantization of foliations, Proceedings of the XXth International Conference on Differential Geometric Methods in Theoretical Physics, 1991, New York City, Vols. 1-2, World Scientific (Singapore), 471-487 (1992). MR1225136


\bibitem[C]{C} Connes, A.; Non-commutative differential geometry, {IHES Publ.
Math., 62 (1985), 257--360}.

\bibitem[CYZ]{CYZ} Chen, X., Yang, S., Zhou, G.
;Batalin-Vilkovisky algebras and the noncommutative Poincaré duality of Koszul Calabi-Yau algebras, J. Pure Appl. Algebra 220 (2016), no. 7, 2500-2532.

\bibitem[G]{G} Gerstenhaber, M.; The cohomology structure of an associative ring, Ann. of Math. 78 (1963) 267-288.

\bibitem[GJ]{GJ} Getzler, E., Jones, J.; The cyclic homology of crossed product
algebras, {J. Reine Angew. Math., 445 (1993), 161-174}.

\bibitem[HT]{HT} Halbout, G., Tang, X.; Noncommutative Poisson structures on orbifolds, {Trans. Amer. Math. Soc., 362 (2010), 2249--2277}.

\bibitem[K]{K} Kontsevich, M.; Deformation quantization of Poisson manifolds, I, Lett. Math. Phys. 66, 157-216 (2003).

\bibitem[LW1]{LW1} Lowen, W, Van den Bergh, M.; Deformation theory of abelian categories, Trans. Amer. Math. Soc. 358 (2006), no. 12, 5441-5483.

\bibitem[LW2]{LW2} Lowen, W, Van den Bergh, M.; Hochschild cohomology of abelian categories and ringed spaces, Adv. Math. 198 (2005), no. 1, 172-221.

\bibitem[M]{M} Le Meur, P.; Crossed products of Calabi-Yau algebras by finite groups; J. Pure Appl. Algebra 224 (2020), no. 10.

\bibitem[NPPT]{NPPT} Neumaier, N., Pflaum, M. J., Posthuma, H., Tang, X.; Homology of formal deformations of proper étale Lie groupoids. J. Reine Angew. Math. 593 (2006), 117-168.


\bibitem[N]{N} Nest, R.; Cyclic cohomology of noncommutative tori. Canad. J. Math. 40 (1988), no. 5, 1046-1057.

\bibitem[NW]{NW} Negron, C, Witherspoon, S.; The Gerstenhaber bracket as a Schouten bracket for polynomial rings extended by finite groups. Proc. Lond. Math. Soc. (3) 115 (2017), no. 6, 1149-1169.

\bibitem[PPTT]{PPTT} Pflaum, M., Posthuma, H., Tang, X., Tseng H.;. Orbifold cup products and ring structures on Hochschild cohomologies. Commun. Contemp. Math., 13(01):123–182, 2011.

\bibitem[Q1]{Q1} Quddus, S.; Hochschild and cyclic homology of the crossed product of algebraic irrational rotational algebra by finite subgroups of $SL_2(\mathbb Z)$, {J. Algebra 447 (2016), 322-366}.

\bibitem[Q2]{Q2} Quddus, S.; cohomology of $\mathcal A_\theta^{alg} \rtimes \mathbb Z_2$ and its Chern-Connes pairing, J. Noncommut. Geom. 11 (2017), no. 3, 827-843. MR3713006

\bibitem[Q3]{Q3} Quddus S.; Cyclic cohomology and Chern Connes pairing of some crossed product algebras, J. Algebra 481 (2017), 120-157. MR3639470

\bibitem[R]{R} Rieffel, M.; $C^\ast$-algebras associated with irrational rotations. Pacific J. Math. 93 (1981), no. 2, 415-429.

\bibitem[SW1]{SW1} Shepler, A., Witherspoon, S.; Finite groups acting linearly: Hochschild cohomology and the cup product, Adv. Math. 226 (4) (2011) 2884–2910. MR2764878

\bibitem[SW2]{SW2} Shepler, A., Witherspoon, S.; Group actions on algebras and the graded Lie structure of Hochschild cohomology, J. Algebra 351 (2012) 350–381. MR2862214

\bibitem[T]{T} Tulczyjew, W.; The graded Lie algebra of multivector fields and the generalized Lie derivative of forms, Bull. Acad. Polon. Sci., 22 (1974), 937-942.


\bibitem[V]{V}  Van den Bergh, M; A relation between hochschild homology and cohomology for georenstein rings; Proc. Amer. Math. Soc. 126 (1998) 1345-1348.


\bibitem[W1]{W1} Wambst, M.; Complexes de Koszul quantiques. (French) [Quantum Koszul complexes] Ann. Inst. Fourier (Grenoble) 43 (1993), no. 4, 1089-1156.

\bibitem[W2]{W2} Wambst, M.; Hochschild and cyclic homology of the quantum multiparametric torus. J. Pure Appl. Algebra 114 (1997), no. 3, 321-329. 

\bibitem[WZ]{WZ} Witherspoon, S., Zhou, G.; Gerstenhaber brackets on Hochschild cohomology of quantum symmetric algebras and their group extensions. Pacific J. Math. 283(1), 223–255 (2016) MR3513848

\bibitem[Xu]{Xu} Xu, P; Noncommutative Poisson algebras, Amer. J. Math. 116, 101-125 (1994). MR1262428

\end{thebibliography}
\end{document}